\documentclass[12pt,A4,reqno]{amsart}

\usepackage{amssymb}
\usepackage{amsmath,amscd}
\usepackage{color}
\usepackage{epsf}
\usepackage{graphicx}

\theoremstyle{plain}
\newtheorem{thm}{Theorem}[section]
\newtheorem{lem}[thm]{Lemma}

\theoremstyle{definition}
\newtheorem{exam}{Example}
\newtheorem{rem}{Remark}

\newcommand{\R}{\mathbb R}
\newcommand{\Z}{\mathbb Z}

\begin{document}

\title [\ ] {On $3$-manifolds with locally standard $(\Z_2)^3$-actions}

\author{Zhi L\"u}
\address{Institute of Mathematics, School of Mathematical Sciences, Fudan
 University, Shanghai, 200433, P.R.China.}
\email{zlu@fudan.edu.cn}
\author{Li Yu}
\address{Department of Mathematics and IMS, Nanjing University, Nanjing, 210093, P.R.China}
\email{yuli@nju.edu.cn}

\thanks{Supported by grants from NSFC (No. 10671034).
 }




\begin{abstract}
   As a generalization of Davis-Januszkiewicz theory, there
is an essential link between locally standard $(\Z_2)^n$-actions (or
$T^n$-actions) actions and nice manifolds with corners, so that a
class of nicely behaved equivariant cut-and-paste operations on
locally standard actions can be carried out in step on nice
manifolds with corners. Based upon this, we investigate what kinds
of closed manifolds admit
    locally standard $(\Z_2)^n$-actions; especially for the $3$-dimensional
    case.
   Suppose $M$ is an orientable closed connected $3$-manifold.  When $H_1(M;\Z_2)=0$, it is shown
   that $M$ admits a locally standard $(\Z_2)^3$-action if and only
   if $M$ is homeomorphic to a connected sum $\overset{\mathrm{8\ copies}}{\overbrace{N \# N \# \cdots \# N
   }}$ of 8 copies of some $\Z_2$-homology sphere $N$, and if further assuming $M$ is
   irreducible, then $M$ must be homeomorphic to $S^3$.
    In addition, the argument is extended to rational homology $3$-sphere $M$ with
    $H_1(M;\Z_2) \cong \Z_2$ and an additional assumption that the $(\Z_2)^3$-action
     has a fixed point.
\end{abstract}

\maketitle

 \section{Introduction}\label{s1}

  During the last two decades, the theory of toric varieties has
  produced a strong pervasion among many mathematical areas, such as
  symplectic geometry, commutative algebra, toric topology etc.,
  because toric varieties admit equivalent descriptions arising naturally in those
  areas. For example, see~\cite{BP02}, \cite{Si01}.  In $1991$, Davis and Januszkiewicz~\cite{DaJan91} introduced and studied two kinds of
  topological versions of toric varieties: small covers and (quasi-) toric
  manifolds. These two kinds of $G$-manifolds ($G=(\Z_2)^n$ or $T^n$) have two special
  properties: the group action is locally standard and the orbit
  space is a simple convex polytope.
Generally, the local standardness of the $G$-action on $M$ makes
sure that the orbit space $Q=M\slash G$ is a nice manifold with
corners. Furthermore, when $Q$ is assumed to be a simple convex
polytope $P$, as shown in~\cite{DaJan91}, a lot of important
algebraic topological invariants for $M$ can be completely computed
in terms of combinatorics of the polytope $P$.
 Indeed, let $\pi: M \longrightarrow P$
be a small cover or a quasi-toric manifold over a simple convex
polytope $P$. Then there is a natural map (called characteristic
function) $\lambda$ defined on facets of $P$, so that up to
equivariant homeomorphism, $M$ can be recovered from the pair
$(P,\lambda)$ in a canonical way (called \textit{glue-back
construction}). As a result, the topological invariants of $M$ are
also completely decided by $(P,\lambda)$.

\vskip .2cm

 The glue-back construction of $M$ provides an effective way of
studying small covers and quasi-toric manifolds. Following this
idea, a description of topological types of all $3$-dimensional
small covers is given in~\cite{LuYu07} by using six kinds of
cut-and-paste operations.

\vskip .2cm

 Moreover, it is shown in~\cite{MaLu08} that the glue-back
construction can be carried out in a by far wider class than simple
convex polytopes. Actually, if we only assume the $G$-action on $M$
is local standard, the orbit space would be a nice manifold $Q$ with
corners which also admits a characteristic function $\lambda$ when
$\partial Q\neq \varnothing$. However, in this more general setting,
$(Q,\lambda)$ is not enough to recover the manifold $M$. Actually we
 need an additional data
--- a principal bundle over $Q$. Then, as shown in~\cite{MaLu08}, up
to equivariant homeomorphism $M$ can be recovered from the
characteristic function $\lambda$ and a principal bundle $\xi$ over
$Q$, and moreover,  a necessary and sufficient condition for two
such manifolds over $Q$ to be equivariantly homeomorphic is given in
terms of $\lambda$ and $\xi$.

\vskip .2cm

 In addition, it is well known that a general effective action may have very
complicated orbit types and the structure of fixed point set, which
are very difficult to be visualized. However, as a special kind of
effective actions, locally standard actions can be well understood
by the glue-back construction, because the structure of orbit types
and fixed point set can be clearly visualized from the manifold with
corner structure of the orbit space. Now, a natural question is

\vskip .2cm

\begin{enumerate}
\item[\textbf{($\star$)}] {\em What kinds of closed manifolds admit
     locally standard $G$-actions?}
\end{enumerate}

\vskip .2cm

The purpose of this paper is to investigate the question $(\star)$
for $G=(\Z_2)^3$, so we restrict our attention to locally standard
$({\Bbb Z}_2)^3$-actions. We shall see that although closed
$3$-manifolds admitting locally standard $({\Bbb Z}_2)^3$-actions
form a wider class than $3$-dimensional small covers, there are
actually very few such irreducible $3$-manifolds with small first
${\Bbb Z}_2$-homology. We first consider $\Z_2$-homology
  $3$-spheres (i.e., closed connected 3-manifolds $M^3$ with
  $H_1(M^3;\Z_2)=0$), and obtain the following result.
 \begin{thm} \label{main1}
     Suppose that $M$ is a  $\Z_2$-homology
  $3$-sphere. Then $M$ admits a locally standard $(\Z_2)^3$-action if
  and only if $M$ is homeomorphic to a connected sum $\overset{\mathrm{8\ copies}}{\overbrace{N \# N \# \cdots \# N
         }}$ of 8 copies of some $\Z_2$-homology sphere $N$. In particular, if
    $M$ is  irreducible and admits
    a locally standard $(\Z_2)^3$-action, then $M$ must be homeomorphic to $S^3$.
   \end{thm}

 Although there are infinitely many irreducible $\Z_2$-homology
  $3$-spheres, Theorem~\ref{main1} asserts that among them only $S^3$ admits a locally standard
  $(\Z_2)^3$-action.

 \begin{rem}
   There is an open conjecture saying that an irreducible homotopy $3$-sphere which
   admits an involution must be homeomorphic to
   $S^3$ (see Kirby's list~\cite[Problem 3.1(E)]{Kirby78}). In our case, we actually show that
    an irreducible homotopy $3$-sphere which
   allows a locally standard $(\Z_2)^3$-action must be homeomorphic to
   $S^3$. So it should be interesting to study the existence
    of locally standard $(\Z_2)^3$-actions on a homotopy $3$-sphere.
 \end{rem}

We also consider orientable
    rational homology $3$-spheres $M$ with $H_1(M;\Z_2) \cong \Z_2
    $, and get the following.

 \begin{thm} \label{main2}
      Suppose that $M$ is an orientable
    rational homology $3$-sphere with $H_1(M;\Z_2) \cong \Z_2 $. Then $M$ admits a locally standard
    $(\Z_2)^3$-action  having a fixed point if and only if $M$ is
    homeomorphic to a connected sum $\R P^3 \# \overset{\mathrm{8\ copies}}{\overbrace{N \# N \# \cdots \# N }}
     $ of a real projective 3-space $\R P^3$ and 8 copies of some  $\Z_2$-homology sphere $N$.
    \end{thm}

  This paper is organized as follows. In Section~\ref{s2} we review some basic notions for  locally
standard $(\Z_2)^n$-actions and manifolds with corners, and study
the relationships between  locally standard $(\Z_2)^n$-actions and
manifolds with corners. In addition, we introduce the general
glue-back construction and the equivariant cut-and-paste operation.
 In Section~\ref{s3}, we list some known results on
  closed manifolds with a finite group action.
 Then we finish the proofs of our main results in
 Section~\ref{s4}.  \\

\section{Locally standard $(\Z_2)^n$-actions and equivariant cut-and-paste
operation}\label{s2}

First, we review some basic notions with respect to  locally
standard $(\Z_2)^n$-actions and manifolds with corners, and then
introduce some useful lemmas.

\subsection{Locally standard $(\Z_2)^n$-actions}
A standard representation of $(\Z_2)^n$ on ${\Bbb R}^n$ is the
natural action defined by
$$\big((g_1,...,g_n), (x_1,...,x_n)\big)\longmapsto \big(
(-1)^{g_1}x_1,..., (-1)^{g_n}x_n\big),$$ which exactly fixes the
origin of ${\Bbb R}^n$ and has the positive cone $\R_{\geq 0}^n= \{
(x_1,\cdots,x_n)\in \R^n
  \ |\ x_i\geq 0 , i=1,\cdots, n  \}$ as
its orbit space.
    An effective action of $(\Z_2)^n$ on an $n$-dimensional closed manifold $M^n$
     is said to be \textit{locally standard} if it locally looks
     like the standard representation of $(\Z_2)^n$ on ${\Bbb R}^n$;
     more precisely, if
     for each point $x$ in $M^n$, there
     is a $(\Z_2)^n$-invariant neighborhood $V_x$ of $x$ such that $V_x$ is
   equivariantly homeomorphic to an invariant open subset of the
   standard $(\Z_2)^n$-representation  on ${\Bbb R}^n$ (see \cite{DaJan91}).

   \vskip .2cm

   Recall from~\cite{Da83}  that an \textit{$n$-manifold with corners} $Q$ is a Hausdorff space together with a maximal
  atlas of local charts onto open subsets of $\R_{\geq 0}^n $
  such that the overlap maps are homeomorphisms of preserving codimension, where
  the codimension $c(x)$ of a point $x=(x_1,\cdots,x_n)$ in $\R_{\geq 0}^n$ is the number of
  $x_i$ that is $0$. Some examples of manifolds with corners are shown in Figure~\ref{p:Corner_Manifolds}.
   \begin{figure}
  \includegraphics[width=0.90\textwidth]{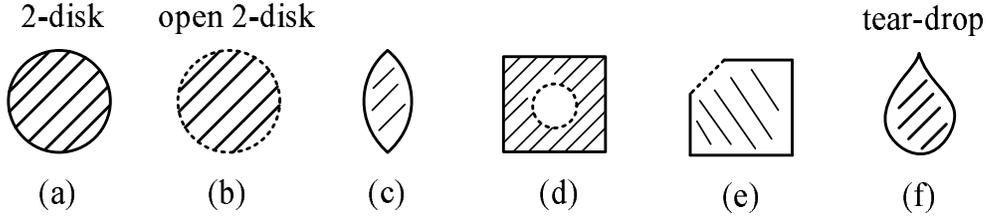}
  \caption{Examples of manifolds with corners}\label{p:Corner_Manifolds}
  \end{figure}
  Generally, a  manifold $Q$ with corners
  may be either compact or non-compact. If  $Q$  is non-compact, then $Q$ is said to be
  an \textit{open} manifold with corners. For example, given a polygon $P^2$,  if we remove a closed 2-disk in its interior
   or cut out a vertex, then the resulting space is an open $2$-manifold with
   corners (see Figure~\ref{p:Corner_Manifolds} (d) and (e)).

\vskip .2cm

  Let $Q$ be a manifold with corners. The \textit{boundary} $\partial Q$ of
  $Q$ is defined as the boundary of $Q$ as a topological manifold.
  Obviously, the boundary of $Q$ is empty if and only if the codimension of each
  point in $Q$ is always zero. Now suppose that $Q$ has non-empty boundary.
  An \textit{open pre-face} of $Q$ of codimension $m$
  is a connected component of $c^{-1}(m)$. A \textit{closed pre-face}
  is the closure of an open pre-face. A closed pre-face of codimension $1$  is called
  a \textit{facet} of $Q$.
  For any $x\in Q$, let $\Sigma(x)$
  be the set of facets which contain $x$.

  \vskip .2cm

  A manifold $Q$ with corners  is said to be \textit{nice} if either $\partial Q$ is empty or $\partial Q$ is non-empty but $\mathrm{Card}(\Sigma(x)) = 2$
  for any $x$ with $c(x) = 2$. It is easy to see that if  an $n$-dimensional nice manifold $Q$ with corners has empty boundary,
  then it is a closed or open topological manifold; otherwise, it has facets, and any $l$-dimensional
  closed pre-face $F^l$ is a component of the intersection of $n-l$
  facets in $Q$, and $\partial Q$ is the union of all
  facets. The Figure~\ref{p:Corner_Manifolds} (f) is not
  nice.
  \vskip .2cm

A subset $K$ of an  $n$-manifold $Q$ with corners is said to be an
$n$-dimensional \textit{ open submanifold with corners}  of $Q$ if
the restriction to $K$ of the atlas of $Q$ makes $K$ become an
 open $n$-manifold with corners in its own right. Obviously,
the intersection of $Q-K$ and the closure $\bar{K}$ is also an
$(n-1)$-manifold with corners, which is called the \textit{section}
of $Q-K$ and is denoted by $S(K)$.

\vskip .2cm Let $Q$ be a nice manifold with corners, and let $K$ be
an open submanifold with corners of $Q$ .  If $S(K)$ is a nice
manifold with corners, then $K$ is said to be a \textit{nice} open
submanifold with corners of $Q$  (see
Figure~\ref{p:open_Submanifold} for examples).

\begin{figure}
         \input{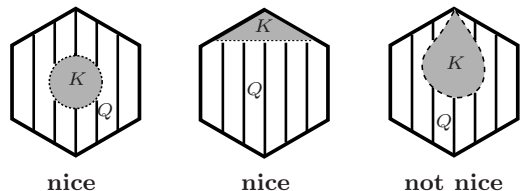}\centering
          \caption{Open submanifold $K$ with corners in a manifold $Q$
           with corners}\label{p:open_Submanifold}
        \end{figure}

\vskip .2cm

  A continuous map between two manifolds with corners
  $f: Q_1 \rightarrow Q_2$ is called a \textit{facial map} if it preserves
  codimension of each point, i.e.,  $c(f(x))=c(x)$ for $\forall\,x\in Q_1$.
  In particular, if $f$ is a homeomorphism, we call it a \textit{facial
  homeomorphism}.

\vskip .2cm
      Suppose $M^n$ is an $n$-dimensional closed manifold with a locally standard
      $(\Z_2)^n$-action. Let $\pi : M^n \rightarrow Q^n$ be the orbit
      map. Then  $Q^n$ is a compact nice
     $n$-dimensional manifold with corners (see also~\cite{Da83} and~\cite{Janich68}).
     In particular, $M^n$ is called a small cover if $Q^n$ is a
    simple convex polytope (see~\cite{DaJan91}).

     \vskip .2cm
    If the boundary of $Q^n$ is empty, then $Q^n$ is a closed manifold,
    so each open neighborhood of any point $x$ in $Q^n$ is identified
     with one in the interior of ${\Bbb R}^n_{\geq 0}$. Since the action
    on $M^n$ is locally standard and $M^n$ is compact, one has that the
    action on $M^n$ must be free, so $\pi : M^n \rightarrow Q^n$ become
     a principal $(\Z_2)^n$-bundle. Conversely, if the action on $M^n$ is
     free, it is easy to see that the boundary of $Q^n$ is empty.

\vskip .2cm

   Now assume that $Q^n$ has non-empty boundary. Then $\partial Q^n$
    is the union of all its facets.
   Similar to the theory of small covers,  each facet $F$ of $Q^n$
    corresponds to a non-zero element $g_F\in (\Z_2)^n$, such that
    $\pi^{-1}(F)$ is fixed by the action of $g_F$ on $M^n$.
    So a characteristic function $\lambda$ on $Q^n$ can be defined as
    \begin{align*}
       \lambda :  \{ \text{facets of}\ Q^n \} &\longrightarrow  (\Z_2)^n \\
                            F  \quad\;   &\longmapsto  \ g_F
    \end{align*}
    satisfying the condition that whenever the intersection $\bigcap_i F_i \neq \varnothing
    $ of some facets is non-empty, all elements of the
    set $\{ \lambda(F_i) \}$ are linearly independent in $(\Z_2)^n$.
    Here $\lambda$ is also called a \textit{$(\Z_2)^n$-coloring} of $Q^n$.

\vskip .2cm Combining the above argument, we have

   \begin{lem} \label{free}
      Suppose that $M^n$ is a closed $n$-manifold admitting a locally standard
      $(\Z_2)^n$-action, and $\pi : M^n \rightarrow Q^n$ is  the orbit
      map. Then $Q^n$ is a compact nice manifold with corners such that
       \begin{enumerate}
       \item[(1)] $\partial Q^n$ is empty if and only if the $(\Z_2)^n$-action on $M^n$ is
       free (i.e.,  $\pi : M^n \rightarrow Q^n$ is a principal
       $(\Z_2)^n$-bundle).
\item[(2)] If $\partial Q^n$ is non-empty, then $Q^n$ admits a $(\Z_2)^n$-coloring on
  its facets.\\
\end{enumerate}
    \end{lem}

            \begin{rem}
If the locally standard $(\Z_2)^n$-action on $M^n$ is not free, then
it is easy to see that the boundary of $Q^n$ together with the
$(\Z_2)^n$-coloring on its facets gives the information of the
non-free orbits, while the interior of $Q^n$ corresponds to all free
orbits.
   \end{rem}

     \begin{lem}\label{L:action2} Suppose that $M^n$ is an orientable
        closed connected manifold admitting  a non-freely locally standard $(\Z_2)^n$-action
        and $\lambda$ is the $(\Z_2)^n$-coloring on its orbit space $Q^n$.  Then
        each
         $\tau \in \text{\rm Im} \lambda$ is  an orientation-reversing involution on $M^n$.
       \end{lem}
\begin{proof} By the definition of $\lambda$, there exists a facet $F$ such that $\tau=\lambda(F)$.
Since $F$ is $(n-1)$-dimensional, $\pi^{-1}(F)$ is an
$(n-1)$-dimensional closed manifold and is fixed by the involution
action of $\tau$ on $M^n$, where $\pi: M^n\longrightarrow Q^n$ is
the orbit map.
        Since the
        $(\Z_2)^n$-action on $M^n$ is locally standard and $M^n$ is orientable, $\pi^{-1}(F)$ separates $M^n$. So $\tau$ must be orientation-reversing.
      \end{proof}

      \begin{rem}
We see from Lemma~\ref{free} and Lemma~\ref{L:action2} that the
local standardness of the action implies many special properties.
Indeed, a general effective action could be much more complicated
than locally standard actions. Here, we give an example of
non-locally standard $(\Z_2)^n$-action below which
     helps us to understand this difference. Consider the
     $(\Z_2)^2$-action  on the unit
     sphere $S^2=\{ (x,y,z)\in \R ^3\ | \ x^2+y^2+z^2=1 \}$ given by
      two commutative involutions $\tau_1$ and $\tau_2$,
  where $\tau_1$ is the reflection of $S^2$ about the $xy$-plane,
     $\tau_2$ is the antipodal map. Obviously,  this  $(\Z_2)^2$-action in
     a small neighborhood of the north pole and the south pole
     is not locally standard.
      \end{rem}

     The following is a typical example for locally standard
$(\Z_2)^n$-actions.
\begin{exam}\label{ex1}
The standard linear action of $(\Z_2)^n$ on $S^n$ defined by
$$\big((g_1,...,g_n), (x_0,x_1,...,x_n)\big)\longmapsto \big(x_0,
(-1)^{g_1}x_1,..., (-1)^{g_n}x_n\big)$$ is locally standard and
fixes the north pole and the south pole of $S^n$, but its orbit
space is not a simple convex polytope except for $n=1$. When $n=3$,
the orbit space of this standard action is the solid three-edged
football \textbf{B}, whose boundary consists of two vertices, three
edges and three $2$-gons.
\end{exam}

\subsection{Reconstruction of locally standard $(\Z_2)^n$-actions}
  Suppose that $M^n$ is a closed manifold with a locally standard
  $(\Z_2)^n$-action, and $\pi: M^n\longrightarrow Q^n$ is its orbit
   map. Assume $Q^n$ is connected in the following discussions.
        \vskip .2cm

        If the boundary of $Q^n$ is empty, then by Lemma~\ref{free} the action on $M^n$ is free, so $\pi: M^n\longrightarrow
        Q^n$ is actually a principal $(\Z_2)^n$-bundle over $Q$. In
        other words, $M^n$ is a $2^n$-fold regular covering space over $Q^n$.
        In this case, there is no canonical way of recovering $M^n$
        from $Q^n$. But up to homeomorphism, $M^n$ corresponds to an
        element of $H^1(Q^n;(\Z_2)^n)$.

        \vskip .2cm

        If $Q^n$ has non-empty boundary, then $Q^n$ admits a
        $(\Z_2)^n$-coloring $\lambda$ decided by the $(\Z_2)^n$ action.
        Generally speaking, as shown in~\cite{MaLu08}, the pair
        $(Q^n, \lambda)$ is not enough for us to recover $M^n$. To do that, one must add another
        data $\xi: E\longrightarrow Q^n$,  which is a principal $(\Z_2)^n$-bundle over $Q^n$.
        This bundle is directly associated with $M^n$ and can be
        produced in the following way: take a facet $F$ of $Q^n$,
        one can obtain a closed submanifold $\pi^{-1}(F)$ of $M^n$.
        The required bundle is given by removing the union of
        small invariant tubular neighborhoods of all these
         $\pi^{-1}(F)$ in $M^n$. In particular,  it is easy to see
         from~\cite{Da83} and~\cite{Janich68} that
        such a bundle is uniquely determined up to isomorphism.

        \vskip .2cm

        Conversely, from the $(\Z_2)^n$-coloring $\lambda$ and the principal bundle
        $\xi$ over $Q$, we can reconstruct  $M^n$
         as follows (also see
        ~\cite{MaLu08}). First, define an equivalence $\sim$ on $E$: for $x_1, x_2\in E$,
        $$x_1\sim x_2\Longleftrightarrow\begin{cases}
        \xi(x_1)=\xi(x_2)\in \text{Int}(Q^n)=Q^n-\partial Q^n\text{ or }
        \\
 \xi(x_1)=\xi(x_2)\in \partial Q^n \text{ and $x_1=gx_2$
for some $g\in G_F$}\end{cases}$$ where $F$ is the closed pre-face
of $Q^n$ such that $\xi(x_1)=\xi(x_2)$ is contained in the relative
interior of $F$ (note that there must be some facets
$F_{i_1},...,F_{i_r}$ of $Q^n$ such that $F$ is a component of the
intersection $F_{i_1}\cap\cdots\cap F_{i_r}$), and $G_F$ is the
subgroup determined by $\lambda(F_{i_1}),...,\lambda(F_{i_r})$.
Next, up to equivariant homeomorphism, the quotient space
$E/\sim=M(\lambda,\xi)$ reproduces $M^n$, and it is called the
\textit{glue-back construction} of $M^n$. In particular, since $Q^n$
is connected and has non-empty boundary, it is easy to see that the
glue-back construction $M(\lambda, \xi)$ is also connected, so is
$M^n$. It should be pointed out that all possible locally standard
$(\Z_2)^n$-actions with $Q^n$ as
orbit space can be constructed in the above way. \\

\begin{exam}
Let $Q$ be the solid three-edged football $\textbf{B}$ in
Example~\ref{ex1}. Clearly $\textbf{B}$ admits a
$(\Z_2)^3$-coloring. Since $\textbf{B}$ is contractible, any
principal $(\Z_2)^3$-bundle over $\textbf{B}$ is always trivial.
Thus, given a $(\Z_2)^3$-coloring $\lambda$ on $\textbf{B}$, up to
equivariant homeomorphism we can always construct a unique closed
3-manifold with a locally standard $(\Z_2)^3$-action such that its
orbit space is $\textbf{B}$. Furthermore, it is easy to check that
any closed $3$-manifold with a locally  standard $(\Z_2)^3$-action
whose orbit space is $\textbf{B}$  is weakly equivariantly
homeomorphic to $S^3$ with the standard $(\Z_2)^3$-action.
\end{exam}

\subsection{Equivariant cut-and-paste operation}
    Suppose that $M_1^n$ and $M_2^n$ are two closed manifolds with non-freely locally standard
    $(\Z_2)^n$-actions, and $Q_1$ and $Q_2$ are their corresponding orbit spaces with $(\Z_2)^n$-colorings $\lambda_1$ and $\lambda_2$ respectively.
    A facial map $f: Q_1 \rightarrow Q_2$ is
    said to be \textit{weak color-matching} if there is an automorphism
    $\sigma : (\Z_2)^n \rightarrow (\Z_2)^n$ such that
     $\sigma(\lambda_2(f(F))) = \lambda_1(F)$ for any facet $F$ of
     $Q_1$. In particular, if $\lambda_2(f(F)) = \lambda_1(F)$ for any facet $F$ of
     $Q_1$, then $f$ is said to be \textit{color-matching}.

     \vskip .2cm

        It is shown in~\cite{LuYu07} that all $3$-dimensional
      small covers can be constructed from some basic ones using
      six equivariant cut-and-paste operations. The common idea of
      these operations is to glue two manifolds
      with locally standard $(\Z_2)^n$-actions
      together along a family of $(\Z_2)^n$-orbits of the same type
      in an equivariant way. Using above definitions, we can make this more precise.

      \vskip .2cm

     Suppose that $K_1$ and $K_2$ are nice
     open submanifolds with corners of $Q_1$ and $Q_2$ respectively
      such that $\pi_1^{-1}(K_1)$ is equivariant homeomorphic to
      $\pi_2^{-1}(K_2)$. Note that by the glue-back construction,
      $\pi_i^{-1}(K_i)$ is an open submanifold of $M^n_i$ ($i=1,2$).
        Then we can construct a closed manifold $M^{\#}$ with a locally standard
      $(\Z_2)^n$-action as follows:

      \vskip .2cm

     Let $\widetilde{g}$ be the
      equivariant homeomorphism from  $\pi_1^{-1}(K_1)$ to
      $\pi_2^{-1}(K_2)$. Obviously, $\widetilde{g}$
       determines an equivariant homeomorphism $\widetilde{g}_1$
       from the boundary of $M_1-\pi_1^{-1}(K_1)$ to that of
       $M_2-\pi_2^{-1}(K_2)$. In the orbit spaces, $\widetilde{g}$
       induces a natural homeomorphism
       $g: K_1\rightarrow K_2$ and $\widetilde{g}_1$ determines
       a homeomorphism $g_1: S(K_1) \rightarrow S(K_2)$. In particular,
       if $K_1$ and $K_2$ have non-empty boundaries, then their facets
      naturally inherit the colorings
     from $Q_1$ and $Q_2$, and $g$ is a color-matching facial homeomorphism
     in this case.  If $S(K_1)$ and $S(K_2)$ have non-empty boundaries, then $g_1$ is also
      a color-matching facial homeomorphism. Now, the required manifold
       $M^{\#}$ is defined as
 \[  M^{\#} = (M_1-\pi_1^{-1}(K_1)) \cup_{\widetilde{g}_1} (M_2-\pi_2^{-1}(K_2)), \]
which is obtained by gluing $M_1-\pi_1^{-1}(K_1)$ and
$M_2-\pi_2^{-1}(K_2)$ along their boundaries via $\widetilde{g}_1$.
 $M^{\#}$ naturally admits a locally standard $(\Z_2)^n$-action.
Correspondingly, its orbit space $Q^{\#}= M^{\#}\slash (\Z_2)^n $
       is
    \[ Q^{\#} = (Q_1-K_1) \cup_{g_1} (Q_2-K_2),  \]
    which is obtained from $Q_1-K_1$ and $Q_2-K_2$ by gluing
    $S(K_1)$ and $S(K_2)$ via $g_1$ (see Figure~\ref{p:Cut_Paste}). In addition, the
     manifold with corners structure and the $(\Z_2)^n$-coloring on $Q^{\#}$ is
     defined by:

    \begin{enumerate}
    \item[(a)]when $S(K_1)$ and
    $S(K_2)$ have non-empty boundaries, for
    each $m$-dimensional pre-face $f_1^m$ of
    $S(K_1)$, let $f_2^m=g(f_1^m) \subset S(K_2)$.
    Then there is a unique
     $(m+1)$-dimensional closed pre-face  $F^{m+1}_1\subset Q_1-K_1$ and $F^{m+1}_2\subset Q_2-K_2$
     such that $ f_1^m\subset \partial F_1^{m+1}$ and $ f_2^m\subset \partial F_2^{m+1}$.
     Then $F_1^{m+1}$ and $F_2^{m+1}$ is glued together
     and define a closed pre-face
     $F^{m+1} = F_1^{m+1}\cup_{g_1} F_2^{m+1}$ in
     $Q^{\#}$ where the sectional face $f_i^m$ is removed (see Figure~\ref{p:Cut_Paste}).
     Clearly, $Q^{\#}$ is also a nice manifold with corners since $S(K_1)$ and
    $S(K_2)$ are nice manifolds with corners. The $(\Z_2)^n$-coloring $\lambda^{\#}$
    on  $Q^{\#}$ can be given in the following way: given a facet $F$ of $Q^{\#}$, if $F$ is a facet of $Q_1-K_1$ (or $Q_2-K_2$), then
    $\lambda^{\#}(F)$ is defined as $\lambda_1(F)$ (or
    $\lambda_2(F)$). If $F$ is the facet determined by two facets $F_1\subset Q_1-K_1$ and $F_2\subset
    Q_2-K_2$, then we define
    $\lambda^{\#}(F)=\lambda_1(F_1)=\lambda_2(F_2)$.
    \vskip .2cm

    \item[(b)] when $S(K_1)$ and
    $S(K_2)$ have empty boundaries, they are in the interiors of
    $Q_1$ and $Q_2$ respectively. In this case,  gluing  $(Q_1-K_1)$ and
    $(Q_2-K_2)$ together via $g_1$
    does not touch the corner structures of $Q_1$ and $Q_2$. Then each facet of $Q^{\#}$
    belongs exactly to one of $Q_1$ or $Q_2$, and the $(\Z_2)^n$-coloring $\lambda^{\#}$
    on $Q^{\#}$ is naturally defined by $\lambda_1$ and $\lambda_2$.

  \begin{figure}
  \includegraphics[width=0.70\textwidth]{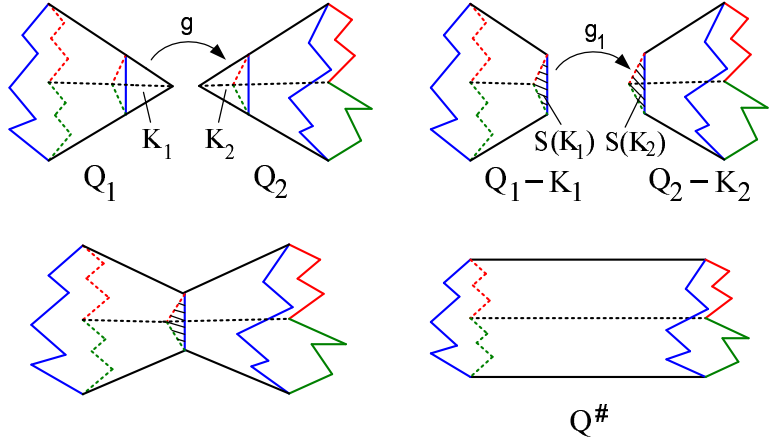}\\
  \caption{Cut-and-Paste operation on orbit spaces}\label{p:Cut_Paste}
  \end{figure}

    \end{enumerate}

    We denote the above procedure by:
    $$(M^{\#}, Q^{\#}, \lambda^{\#})=(M_1, Q_1, \lambda_1)\#_{K_1, K_2} (M_2,
    Q_2, \lambda_2),$$
    and call $(M^{\#}, Q^{\#}, \lambda^{\#})$ the \textit{equivariant
   cut-and-paste operation} of $(M_1, Q_1, \lambda_1)$ and $(M_2,
    Q_2, \lambda_2)$ via $K_1$ and $K_2$.

\vskip .2cm

     As an example, the equivariant connected sum is
      a special equivariant cut-and-paste operation.

\begin{rem}
 Using the glue-back construction,  we can use the
equivariant cut-and-paste operations on the level of orbit spaces to
understand the corresponding equivariant operations on manifolds.
\end{rem}

  \ \\

    \section{Orientation-reversing involution on $3$-manifolds}\label{s3}

 Recall that a connected closed $3$-manifold $M$ is
     called \textit{prime} if $M\approx P \# Q$ implies
     $P\approx  S^3$ or $Q \approx  S^3$.
     $M$ is called \textit{irreducible} if every embedded $2$-sphere in M bounds a $3$-ball
    in $M$. In particular, if $M$ is not a $S^2$-bundle over $S^1$, then $M$ is prime if and only if
    it is irreducible. It is well known that any orientable connected closed $3$-manifold
    $M$ is a connected sum of finitely many prime $3$-manifolds,
      and each connected summand is uniquely determined by $M$ up to homeomorphism.

      \vskip .2cm

      It is well known that not all closed
       $3$-manifolds admit an involution. Indeed, it is shown in~\cite{My81} that there are infinitely
      many irreducible homology $3$-spheres that do not admit any
    \textit{PL}-involutions.

    \vskip .2cm

    On the other hand, if a closed $3$-manifold $M$ is orientable and admits an orientation
   reversing involution $\sigma$, then the fixed point set of $\sigma$
    has some topological restrictions from the homology group of $M$.

  \begin{thm}[Kobayashi~\cite{Koba88}] \label{Eq_Kobe}
    For a closed orientable
  $3$-manifold $M$ which admits an orientation-reversing involution
  $\tau$,
  \begin{equation*}
   \text{dim}_{\Z_2}H_1(\text{Fix}(\tau);\Z_2)\leq\dim_{\Z_2} H_1(M;\Z_2)+\beta_1(M).
   \end{equation*}
   where $\beta_1(M)$ is the first Betti number of $M$.
    \end{thm}

  It is possible that an involution on
   a $3$-manifold has a disconnected fixed point set.
   For example, the fixed point set of the obvious involution on $\R P^3$
   via reflection is the union of an $\R P^2$ and a point.
   \vskip .2cm

   In addition, let us recall some famous theorems on finite group
   actions on manifolds which will be useful in the proofs of
   our main results in section~\ref{s4}.

   \begin{thm}[Borel ~\cite{Borel60}]\label{Borel}
     Suppose that $G=(\Z_p)^k$ acts on an $n$-dimensional $\Z_p$-homology sphere
     $X$. For any subgroup $H$ of $G$, the fixed point set $X^H$ is
     a $\Z_p$-homology sphere of dimension, say, $n(H)$. Then
     $$ n- n(G) = \sum_H\, (n(H) -n(G)), $$ where the sum is over all
     $H$ of corank $1$.\\
   \end{thm}

   \begin{thm}[tom Dieck~\cite{Dieck79} and
   Huang~\cite{Huang74}]\label{DieckHuang}
    For a periodic transformation $f$ on a compact smooth manifold, $L(f)=\chi(\text{Fix}
   (f))$, where $L(f)$ is the Lefschetz number of $f$.\\
   \end{thm}

 \section{Proofs of Theorem~\ref{main1} and Theorem~\ref{main2}}\label{s4}

 The objective of this section is to give proofs of our main results.
 First, let us prove Theorem~\ref{main1}.

 \vskip .2cm

   \noindent \textit{Proof of Theorem~\ref{main1}.}
   With the assumptions on $M$,
   $M$ must be orientable because $H_1(M,\Z_2) =0$. Suppose that $M$ admits a locally-standard
    $(\Z_2)^3$-action. Let $\pi:
   M\longrightarrow Q$ be the orbit map.
    Borel's Theorem~\ref{Borel} implies that
   the $(\Z_2)^3$-action on $M$ can not be a free action (also see
   ~\cite{Dot81} for more general cases), so by Lemma~\ref{free},
   $Q$ has non-empty boundary and admits a $(\Z_2)^3$-coloring
   $\lambda$. Furthermore,  by Lemma~\ref{L:action2}, each  $\tau \in \text{\rm Im}\lambda$
   is an orientation-reversing
   involution on $M$ such that its
    fixed point set $\text{Fix}(\tau)$ has $2$-dimensional
    components.
   In addition, by  Theorem~\ref{Borel}, we have
   $H_1(\text{Fix}(\tau);\Z_2) =0$, so the $2$-dimensional components in the
   $\text{Fix}(\tau)$ must be $2$-spheres (note that this can also be seen from the Smith theory). Moreover,
   by Theorem~\ref{DieckHuang} and the universal coefficient theorem, $\chi(\text{Fix}(\tau))=2$. This implies
   that $\text{Fix}(\tau)$ contains only one  $2$-dimensional component, which is  exactly one
   $S^2$.
   \vskip .2cm
Let $\langle\tau\rangle$ denote the $\Z_2$ subgroup generated by
$\tau$. Since the $(\Z_2)^3$-action on $M$ is locally standard, the
action of $(\Z_2)^3/\langle\tau\rangle$ on the unique
$2$-dimensional component (i.e., $S^2$) of $\text{Fix}(\tau)$ is
still locally standard. Therefore,  the boundary of $Q$  is
tessellated by
    patches that are orbit spaces of some $S^2$'s with locally standard
    $(\Z_2)^2$-actions.

    \vskip .2cm

    \textbf{Claim 1.} {\em The boundary of $Q$ is tessellated by
    $2$-gons.}

    \vskip .2cm

    In fact, if $\Sigma$ is a closed surface with a locally-standard $(\Z_2)^2$-action,
       its orbit space $F$ is a $2$-manifold with corners and $\partial
       F$ is not empty.  We know from~\cite[Corollary 4.2]{MaLu08}
        that
       \begin{equation}\label{Eq:EulerNumber}
         \chi(\Sigma) = 4 \chi(F) - m,
         \end{equation}
         where $m$ is the number of vertices in $F$ (i.e.,  the number of the points fixed by the
       $(\Z_2)^2$-action).
       \vskip .2cm

       In particular when $\Sigma = S^2$, $\chi(\Sigma) =2$.
       Since $\chi(F)\leq 1$, so $\chi(F)=1$ since $\partial F$ is
       not empty. Then $m=2$ and $F$ is homeomorphic to the 2-disk $D^2$, so $F$ must be a $2$-gon.

       \vskip .2cm

Thus, topologically the boundary $\partial Q$ must consist of some
$2$-spheres, each of which  is covered by
      $2$-gon facets such that each vertex of $Q$ meets exactly $3$ facets. So  each boundary component of $Q$
    is   actually  a union of  exactly three
      2-gons,  as shown in Figure~\ref{p:3-Lens}.
        \begin{figure}
  \includegraphics[width=0.1\textwidth]{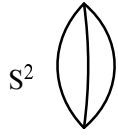}\\
  \caption{$2$-sphere covered by $2$-gons}\label{p:3-Lens}
  \end{figure}

\vskip .2cm
       \textbf{Claim 2.}  {\em $\partial Q$ is connected.}

       \vskip .2cm

       Suppose that $Q$ has $k$ boundary components ($2$-spheres).
      If $k > 1$, then we can do an equivariant cut-and-paste
      operation between $(M, Q, \lambda)$ and $(S^3, \textbf{B},
      \lambda_0)$, which is stated in terms of orbit spaces as
      follows: cut out an  open
          neighborhood $N(C)$ of a boundary component $C$ of $Q$
         and an  open
          neighborhood $N(\partial\textbf{B})$ of the boundary component  of the solid three-edged footbal $\textbf{B}$,
          and then glue them together along
         the new boundaries, as shown in
          Figure~\ref{p:Fill_ball}. \begin{figure}
         \includegraphics[width=0.70\textwidth]{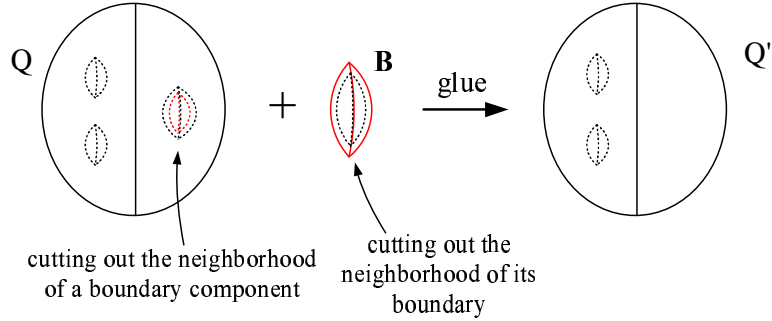}
          \caption{Fill a hole by a $3$-ball}\label{p:Fill_ball}
        \end{figure}
Write $(M', Q',\lambda')=(M, Q,
\lambda)\#_{N(C),N(\partial\textbf{B})}(S^3, \textbf{B},
      \lambda_0)$. On the other hand, this operation is invertible
      with respect to $(M, Q, \lambda)$ and $(M', Q',\lambda')$; in
      other words, we can cut out a small open $3$-ball $B$
         from the interior of $Q'$ and a solid three-edged football $\textbf{B}$ respectively,  and then glue them along
         the new boundaries to get $Q$, as shown in Figure~\ref{p:Glue_Ball}.
         Thus, $(M, Q,\lambda)=(M', Q',\lambda')\#^{-1}_{N(C),N(\partial\textbf{B})}(S^3, \textbf{B},
      \lambda_0)$, where $\#^{-1}_{N(C),N(\partial\textbf{B})}=\#_{B,B}$ is
       the equivariant connected sum along free orbits.  Since $k-1>0$, $\partial Q'$ still contains some $2$-spheres
          as in Figure~\ref{p:Glue_Ball}, so $M'$ is still a connected $3$-manifold.
         We see that $M$ is actually obtained by
          attaching a connected $3$-manifold to $S^3$ by eight disjoint $S^2\times I$ tubes (see
          Figure~\ref{p:Glue_Mfd}). This implies that $\text{rank}H_1(M,\Z)\geq
          7$, which contradicts $H_1(M;\Z_2)=0$. Thus, $k$ must be
          1, so $\partial Q$ is connected.

          \begin{figure}
          \includegraphics[width=0.70\textwidth]{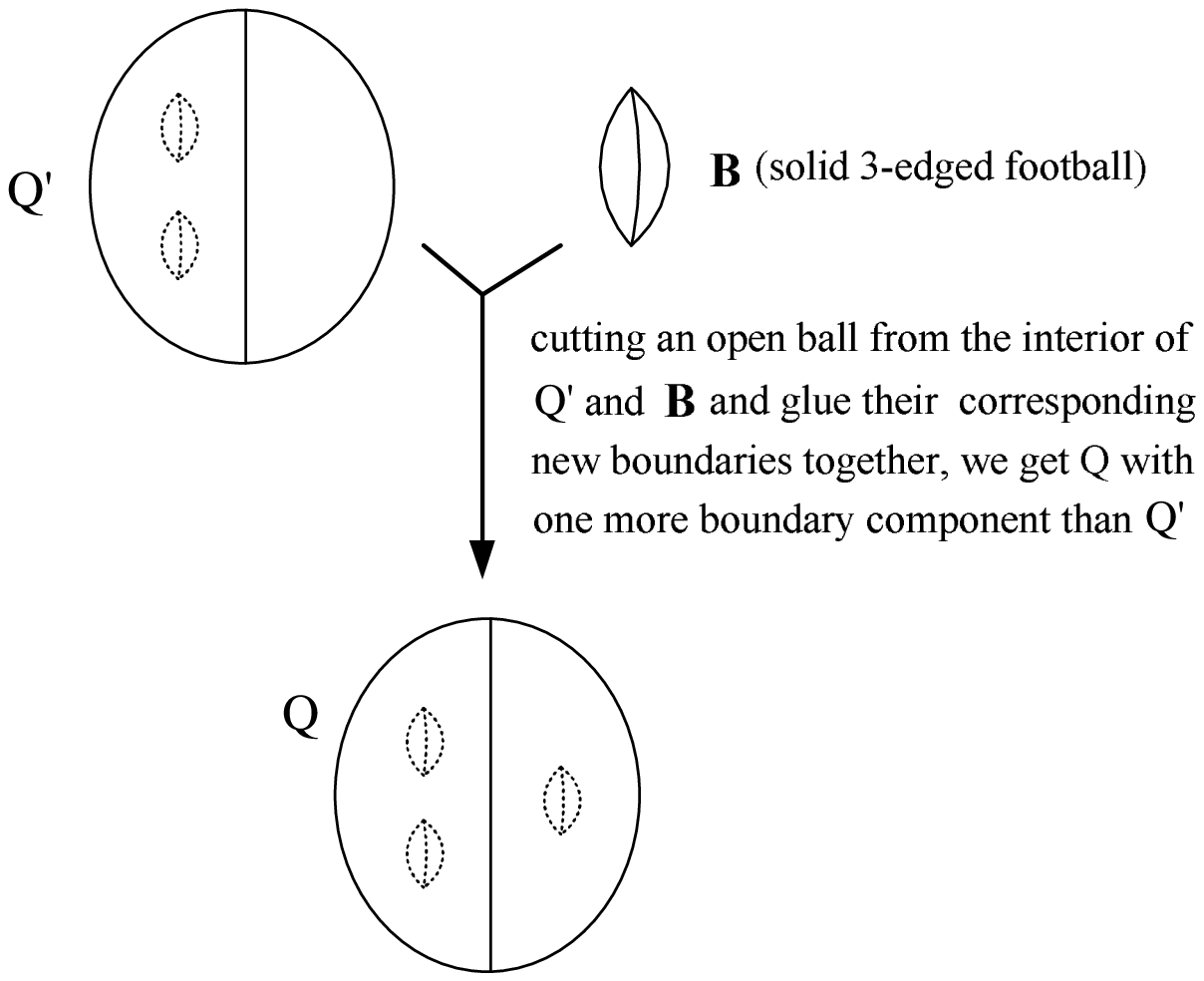}
           \caption{}\label{p:Glue_Ball}
         \end{figure}

           \begin{figure}
         \includegraphics[width=0.43\textwidth]{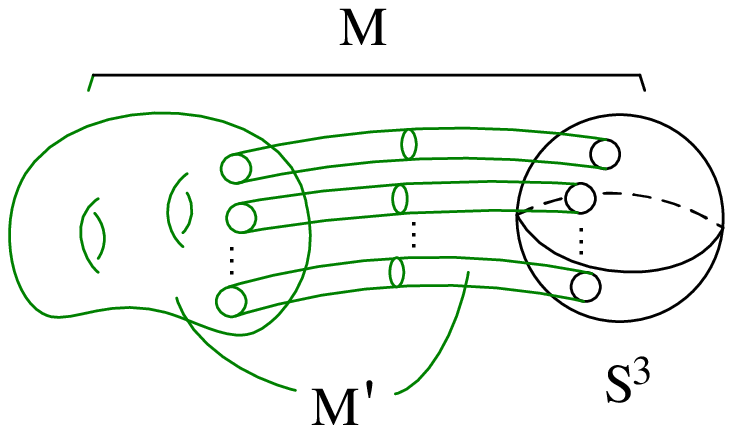}
          \caption{}\label{p:Glue_Mfd}
        \end{figure}

\vskip .2cm

           Now we know that $\partial Q \approx S^2$. We can
           still carry out the  operation $$(M, Q,
\lambda)\#_{N(S^2),N(\partial\textbf{B})}(S^3, \textbf{B},
      \lambda_0)$$ as above, and still denote it by $(M', Q',\lambda')$.  Then it is easy to see that $Q'$  becomes a closed
          manifold!  In this case, all the orbits of the $(\Z_2)^3$-action
          on $M'$ are free orbits. But $M'$ may not be connected. Notice that
          each connected component of $M'$ must be a covering space of $Q'$, whose
          covering cardinality divides $8$. By the equivariance of
          the $(\Z_2)^3$-action, $M'$ must have $1, 2,4$ or $8$ connected components.
         Topologically, the possible relations between $M$ and $M'$ are shown in
         Figure~\ref{p:Glue_Mfd},
         Figure~\ref{p:Glue_Mfd_seimfree} and Figure~\ref{p:Glue_Mfd_free}.
          \begin{figure}
         \includegraphics[width=0.9\textwidth]{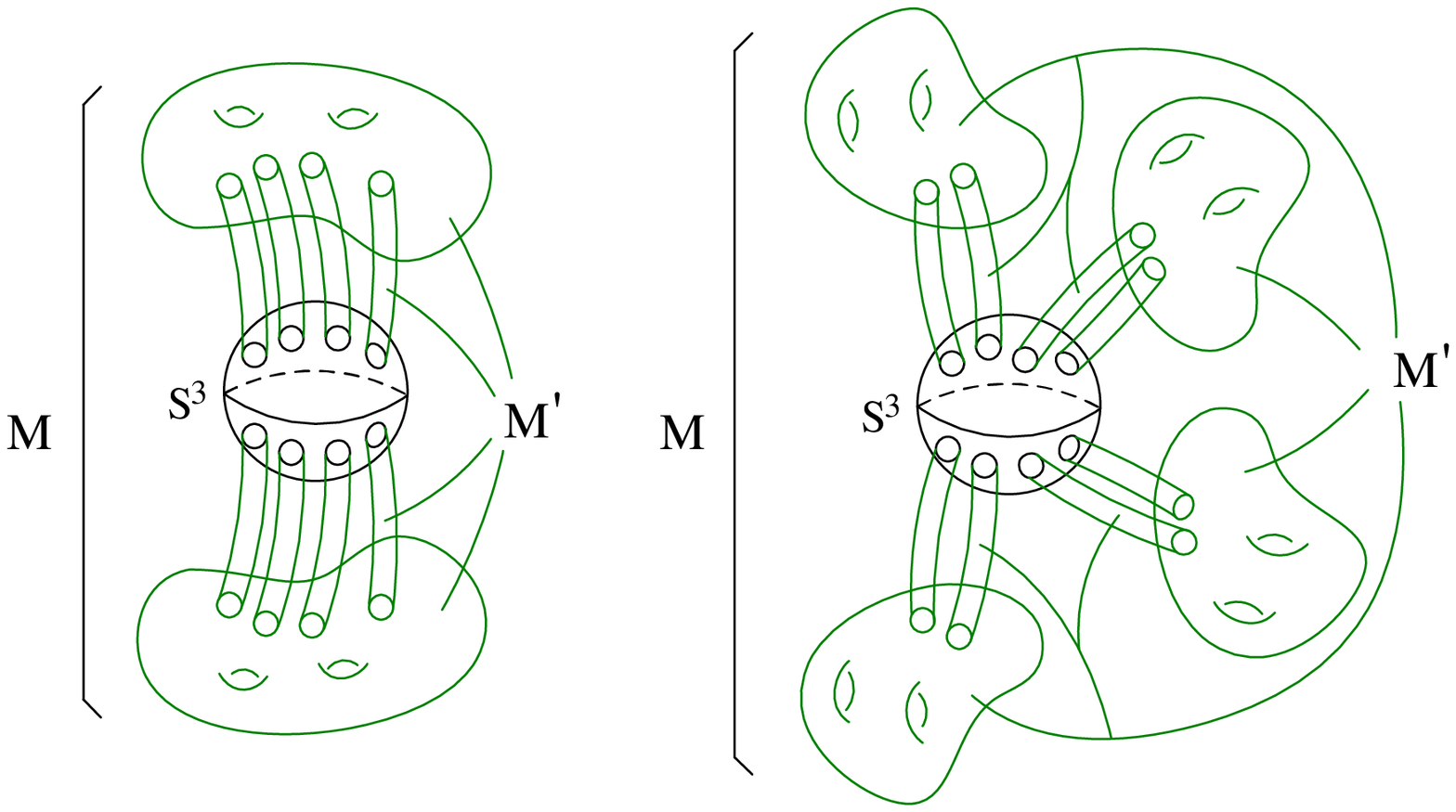}
          \caption{}\label{p:Glue_Mfd_seimfree}
        \end{figure}

            \begin{figure}
         \includegraphics[width=0.5\textwidth]{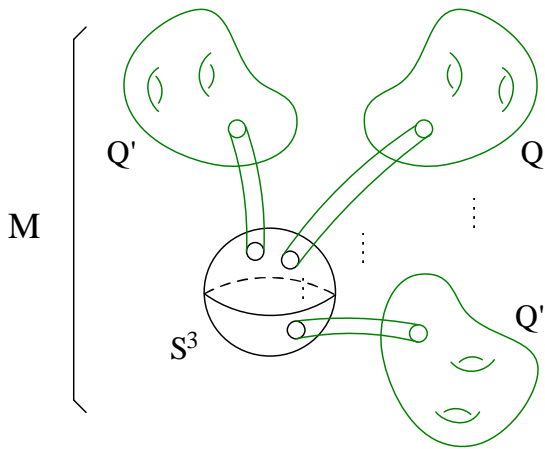}
          \caption{ Eight $Q'$-handles attached to $S^3$ }\label{p:Glue_Mfd_free}
        \end{figure}

        \vskip .2cm

          The cases listed in Figure~\ref{p:Glue_Mfd} and
         Figure~\ref{p:Glue_Mfd_seimfree} are not possible because
         $H_1(M;\Z_2)=0$. So the only possible case is
         Figure~\ref{p:Glue_Mfd_free} where $M$ is the connected sum of
          eight copies of $Q'$ with a $S^3$. This implies that $Q'$ is a $\Z_2$-homology sphere.
          So $M$ is of the form
          $M = \overset{\mathrm{8\ copies}}{\overbrace{N \# N \# \cdots \# N }}
         $, where $N$ is a $\Z_2$-homology sphere. Conversely, a closed $3$-manifold $M$
         of this form obviously admits a locally-standard $(\Z_2)^3$-action.
\vskip .2cm

            If we assume $M$ is irreducible, then $Q'$ in Figure~\ref{p:Glue_Mfd_free} must be
          homeomorphic to $S^3$. So $Q$ is homeomorphic to a
          $3$-ball whose boundary consists of
         $2$-gon facets. Then
          $Q$ must be the solid three-edged football $\textbf{B}$.
        Hence $M$ must be $S^3$ with a locally standard
      $(\Z_2)^3$-action.$\hfill\Box$

\vskip .2cm

Now let us complete the proof of Theorem~\ref{main2}.

\vskip .2cm

\noindent \textit{Proof of Theorem~\ref{main2}.}
     With the conditions on $M$,
if $M$ admits a locally standard
    $(\Z_2)^3$-action  having a fixed point, then the orbit space
    $Q$ of the action has non-empty boundary, so $Q$ admits a
    $(\Z_2)^3$-coloring $\lambda$.
    Furthermore, for each $\tau\in \text{Im}\lambda$, by Lemma~\ref{L:action2},   $\tau$
   is an orientation-reversing
   involution on $M$ such that its
    fixed point set $\text{Fix}(\tau)$ has $2$-dimensional
    components. By Theorem~\ref{Eq_Kobe},
     $$\text{dim}_{\Z_2}H_1(\text{Fix}
     (\tau);\Z_2)\leq\dim_{\Z_2} H_1(M;\Z_2)+\beta_1(M) \leq 1.$$
     So the fixed point set of $\tau$ must be the disjoint union
     of 2-spheres, or real projective planes, or  discrete points.
     Theorem~\ref{DieckHuang} implies that $\text{Fix}(\tau)$ is
     either a $S^2$, a $\R P^2$ plus a point or two points.
     By the equation~\eqref{Eq:EulerNumber}, the orbit space of $\R P^2$
     under a locally-standard $(\Z_2)^2$-action can only be
     a triangle whose vertices are fixed points. So the boundary of $Q$ must consist of
     $2$-spheres, because no other closed surfaces can be covered by $2$-gons and
     $3$-gons with each vertices incident to three different facets (again an easy calculation of the
     Euler number).
     The only possible tessellations of each boundary component of $Q$ in this
     case are (see Figure~\ref{p:2-spheres}):
     \begin{enumerate}
       \item three $2$-gons;
       \item four triangles as the boundary of a regular tetrahedron.

       \begin{figure}
         \includegraphics[width=0.3\textwidth]{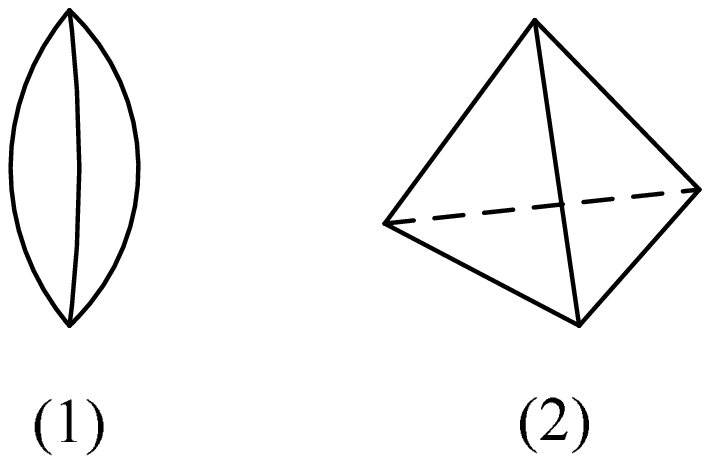}
          \caption{ }\label{p:2-spheres}
        \end{figure}

     \end{enumerate}

      As the orbit spaces of locally-standard $(\Z_2)^3$-actions,
       solid  three-edged  football must correspond to $S^3$ and
      tetrahedron must correspond to $\R P^3$.

      \vskip .2cm

      Assume the orbit space $Q$ has $k$ boundary components.
      Similarly to the proof of Theorem~\ref{main1}, we can show that $k=1$ using
      the condition $H_1(M;\Z_2) \cong \Z_2$.

      \vskip .2cm

       If assuming $M$ is irreducible,
       $Q$ must be a $3$-ball whose boundary is one of the above two cases.
       (The only new ingredient here is
      that the \textquotedblleft $S^3$\textquotedblright\ in the
      Figure~\ref{p:Glue_Mfd}, Figure~\ref{p:Glue_Mfd_seimfree}
      and Figure~\ref{p:Glue_Mfd_free} above could be $S^3$ or $\R P^3$ in this
      case). Considering the homology, $M$ must be $\R P^3$.
      But without the\textquotedblleft irreducibility\textquotedblright
      condition of $M$, $M$ could be homeomorphic to
       \[
        \R P^3 \# \overset{\mathrm{8\ copies}}{\overbrace{N \# N \# \cdots \# N }}
     \]
     where $N$ is a $\Z_2$-homology $3$-sphere.

     Conversely, if $M$ is homeomorphic to the connected sum $\R P^3 \# \overset{\mathrm{8\ copies}}{\overbrace{N \# N \# \cdots \# N
     }}$ with $N$  a $\Z_2$-homology 3-sphere, then obviously $M$
     admits a locally standard $(\Z_2)^3$-action such that the fixed
     point set is non-empty.   $\hfill\Box$

     \vskip .2cm

 \begin{rem}
   General effective $(\Z_2)^n$-actions on cohomology projective
   spaces were investigated in~\cite{Hsiang74}, and some
   general structural results of these actions were obtained there.
 \end{rem}

 \begin{rem} We might be able to remove
 the extra condition in Theorem~\ref{main2} that the locally-standard
 $(\Z_2)^3$-action has a fixed point to arrive at the same
 conclusion. But we must show that there are no free
 $(\Z_2)^3$-actions on a rational homology sphere $M$ with $H_1(M;\Z_2)\cong
 \Z_2$.\\
 \end{rem}

\noindent\textbf{Question:} Is there any free $(\Z_2)^3$-action on a
rational homology $3$-sphere $M$ with $H_1(M;\Z_2)\cong \Z_2$?

\vskip .2cm

 The arguments in above theorems could be extended to investigate
 general $3$-manifolds with locally-standard
 $(\Z_2)^3$-actions. The main difficulty in the general cases
 is that the boundary components of the orbit space of the action might
 have closed surfaces with higher genus, which make this approach
 very complicated.

  \ \\


\begin{thebibliography}{99}

\bibitem{BP02} V. M. Buchstaber and T.E. Panov,
{\em Torus actions and their applications in topology and
 combinatorics}, University Lecture Series, 24.
 American Mathematical Society, Providence, RI, 2002.

 \bibitem{Si01} Ana Cannas da Silva, \textit{Lectures on symplectic
geometry}, Lecture Notes in Mathematics, \textbf{1764},
Springer-Verlag, Berlin, 2001.

\bibitem{DaJan91}  M.~W.~Davis and T.~Januszkiewicz, \textit{Convex polytopes,
Coxeter orbifolds and torus actions}, Duke Math. J. \textbf{62}
(1991), no.\textbf{2}, 417--451.

\bibitem{LuYu07}
 Z.~L\"u and L.~Yu, \textit{Topological types of 3-dimensional small
covers}, arXiv:0710.4496.

\bibitem{MaLu08}
  Z.~L\"u and M.~Masuda, \textit{Equivariant classification
  of $2$-torus manifolds}, arXiv:0802.2313.


\bibitem{Kirby78}
 R.~Kirby, \textit{Problems in low dimensional manifold theory},
 Proceedings of symposia in pure mathematics, Vol. \textbf{32} (1978).

\bibitem{Da83}
 M.~Davis, \textit{Groups generated by reflections and aspherical
  manifolds not covered by Euclidean space}, Ann. of Math. 117 (1983),
  293-324.


 \bibitem{Janich68}
K.~J\"anich, \textit{On the Classification of $O(n)$-Manifoids},
Math. Annalen 176, 53--76 (1968).


\bibitem{My81}
 R.~Myers, \textit{Homology $3$-spheres which admit no PL
 involutions}, Pacific J. Math. \textbf{94} (1981), no. \textbf{2}, 379--384.

 \bibitem{Koba88}
   M.~Kobayashi, \textit{Fixed point sets of orientation reversing
  involutions on $3$-manifolds}, Osaka J. Math. \textbf{25} (1988), no. \textbf{4},
   877--879.

\bibitem{Borel60}
  A.~Borel, \textit{Seminar on transformation groups}, Ann. of Math.
Studies, no. \textbf{46}, Princeton Univ. Press, Princeton, N.J.,
1960.

\bibitem{Dieck79}
  T. tom~Dieck, \textit{Transformation groups and representation
  theory}, Lecture Notes in Math., \textbf{766}, Springer-Verlag,
  Berlin, Heidelberg and New York, 1979.

\bibitem{Huang74}
  W.~H.~Huang, \textit{Equivariant method for periodic maps},
  Trans.~Amer.~Math.~Soc., \textbf{189} (1974), 175-183.


\bibitem{Dot81}
 R.~Dotzel, \textit{Abelian $p$-group actions on homology spheres},
Proc. Amer. Math. Soc. \textbf{83} (1981), no. \textbf{1}, 163--166.

\bibitem{Da78} M. Davis, {\em Smooth $G$-manifolds as
collections of fiber bundles}, Pacific J. Math., {\bf 77} (1978),
315-363.

\bibitem{Hsiang74}
 W. Y. Hsiang, \textit{Structural theorems for topological actions of
$\Z\sb{2}$-tori on real, complex and quaternionic projective
spaces}, Comment. Math. Helv. \textbf{49} (1974), 479-491.


\end{thebibliography}
\end{document}